\pdfoutput=1 
\newif\ifpreprint
\preprinttrue 
\ifdefined\ispreprint
  \preprinttrue
 \fi
\ifdefined\isjournal
  \preprintfalse
\fi

\ifpreprint
\documentclass[a4paper,10pt]{article}
\else

\RequirePackage{fix-cm}
  \documentclass{svjour3}   
\fi

 \ifpreprint
\usepackage{mathrsfs}  
\usepackage[a4paper]{geometry}
\usepackage[utf8]{inputenc}
\usepackage{microtype}
\usepackage[normalem]{ulem}
\usepackage{amsmath,amsthm,amsfonts}
\usepackage{amsmath,amssymb}
\usepackage{thmtools}
\else
\usepackage[hidelinks]{hyperref}
\usepackage{latexsym,amsfonts,amsmath,graphics,multirow}
\usepackage{mathrsfs}
\usepackage[normalem]{ulem}
\usepackage{ragged2e}
\usepackage{epsfig,tikz,pgfplots}
\usetikzlibrary{positioning}
\usepackage{enumitem,amssymb}
\usepackage{algorithm,algorithmicx}
\usepackage{float}
\usepackage{makecell}
\usepackage[font=scriptsize]{caption}

\spnewtheorem{corollary}[theorem]{Corollary}{\bfseries}{\itshape}
\spnewtheorem{definition}[theorem]{Definition}{\bfseries}{\itshape}
\spnewtheorem{assumption}[theorem]{Assumption}{\bfseries}{\itshape}
\spnewtheorem{lemma}[theorem]{Lemma}{\bfseries}{\itshape}
\spnewtheorem{theorem}[theorem]{Theorem}{\bfseries}{\itshape}
\spnewtheorem{proposition}[theorem]{Proposition}{\bfseries}{\itshape}
\spnewtheorem{remark}[theorem]{Remark}{\itshape}{\itshape}
\numberwithin{equation}{section}
\numberwithin{theorem}{section}
\fi
  \ifpreprint
  \numberwithin{equation}{section} 
  \declaretheorem[Refname={Theorem,Theorems}]{theorem}
  \numberwithin{theorem}{section} 
  \declaretheorem[style=definition,numberlike=theorem,Refname={Definition,Definitions}]{definition}
  
  \declaretheorem[style=theorem,numberlike=theorem,Refname={Example,Examples}]{example}
  \declaretheorem[style=definition,numberlike=theorem,Refname={Remark,Remarks}]{remark}
  \declaretheorem[numberlike=theorem,Refname={Lemma,Lemmas}]{lemma}
  \declaretheorem[name=Corollary,numberlike=theorem,Refname={Corollary,Corollaries}]{corollary}
  \declaretheorem[name=Proposition,numberlike=theorem,Refname={Proposition,Propositions}]{proposition}
  \fi
\usepackage{mathtools}

\usepackage[labelfont=bf]{caption}
\usepackage{graphicx}
\usepackage[hidelinks]{hyperref}
\usepackage{todonotes}
\usepackage{tikz,pgfplots} 
\pgfplotsset{compat=1.3}
\usepgfplotslibrary{groupplots}
\usetikzlibrary{calc,trees,positioning,arrows,chains,shapes.geometric,%
    decorations.pathreplacing,decorations.pathmorphing,shapes,%
    matrix,shapes.symbols}
\tikzset{
>=stealth',
  punktchain/.style={
    draw=black, very thick,
    text width=10em, 
    minimum height=3em, 
    text centered, 
    on chain},
  tuborg/.style={decorate},
  tubnode/.style={midway, right=2pt},
}

\newcommand{\bsx}{{\boldsymbol{x}}}

\newcommand{\rd}{\,\mathrm{d}}

\newcommand{\e}{\mathrm{e}} 

\newcommand{\Z}{\mathbb{Z}}
\newcommand{\T}{\mathbb{T}}
\newcommand{\R}{\mathbb{R}}
\newcommand{\N}{\mathbb{N}}

\newcommand{\I}{\mathbb{I}}

\newcommand{\calI}{\mathcal{I}}
\newcommand{\calH}{\mathcal{H}}
\newcommand{\calO}{\mathcal{O}}

\usepackage[hidelinks]{hyperref}

\definecolor{darkgreen}{RGB}{0,130,70}
\definecolor{darkorange}{RGB}{180,60,0}

\newcommand{\Aa}{A^\dagger}  
\newcommand{\Ab}{A^*}       

\title{Construction of Optimal Algorithms for Function Approximation in Gaussian Sobolev Spaces}
\ifpreprint
\author{
Yuya Suzuki\textsuperscript{1} and Toni Karvonen\textsuperscript{2} \vspace{0.5cm}
\\
\kern-1cm\textsuperscript{1}\emph{Department of Mathematics and Systems Analysis, Aalto University,}\\ \emph{ P.O. Box 11100, 00076, Helsinki, Finland} \vspace{0.1cm}
\\
\kern-1cm\textsuperscript{2}\emph{School of Engineering Science, Lappeenranta–Lahti University of Technology,}\\ \emph{Yliopistonkatu 34, 53850 Lappeenranta, Finland}
}
\else

\title{Construction of Optimal Algorithms for Function Approximation in Gaussian Sobolev Spaces}

\titlerunning{Optimal Algorithms for Function Approximation in Gaussian Sobolev Spaces}

\author{
Yuya Suzuki, Toni Karvonen}

\authorrunning{Yuya Suzuki, Toni Karvonen}

\institute{
              Yuya Suzuki\at
              Department of Mathematics and Systems Analysis, Aalto University, P.O. Box 11100, FI-00076 Aalto, Finland\\
			  E-mail:
             {\tt yuya.suzuki@aalto.fi}
             \\
             Toni Karvonen\at
              School of Engineering Science, Lappeenranta–Lahti University of Technology LUT, Yliopistonkatu 34, 53850 Lappeenranta, Finland\\
			  E-mail:
             {\tt toni.karvonen@lut.fi}
}

\date{\today}

\fi

\begin{document}
\maketitle

\begin{abstract}
\noindent
This paper studies function approximation in Gaussian Sobolev spaces over the real line and measures the error in a Gaussian-weighted $L^p$-norm.
We construct two linear approximation algorithms using $n$ function evaluations that achieve the optimal or almost optimal rate of worst-case convergence in a Gaussian Sobolev space of order $\alpha$.
The first algorithm is based on scaled trigonometric interpolation and achieves the optimal rate $n^{-\alpha}$ up to a logarithmic factor.
This algorithm can be constructed in almost-linear time with the fast Fourier transform.
The second algorithm is more complicated, being based on spline smoothing, but attains the optimal rate $n^{-\alpha}$.
\end{abstract}

\section{Introduction}

This paper is concerned with approximating functions in Gaussian-weighted Sobolev spaces over the real line.
We are interested in the \emph{sampling recovery problem}, constructing linear approximation algorithms $A_n(f)$ that recover a function $f \colon \mathbb{R} \to \mathbb{R}$ using $n$ function evaluations.
Let $\rho(x) \coloneqq (2\pi)^{-1/2} \, \e^{-x^2/2}$ be the Gaussian density function and $p \in [1, \infty)$.
We measure the approximation error in the weighted $L^p(\mathbb{R})$ norm
\begin{align} \label{eq:error-measure}
\|f-A_n(f)\|_{L_\rho^p} \coloneqq \left(\int_\R \left|f(x)-[A_n(f)](x)\right|^p \rho(x) \rd x  \right)^{1/p}
\end{align}
and assume that $f$ is an element of the Gaussian Sobolev space
\begin{align} \label{eq:gaussian-sobolev-space}
W^{\alpha,q}_{\rho} \coloneqq \bigg\{f\in L_\rho ^q\; \bigg| \;\|f\|_{W^{\alpha,q}_{\rho}} \coloneqq \bigg(\sum_{\tau=0}^\alpha \int_\R |f^{(\tau)}(x)|^q \rho(x) \rd x  \bigg)^{1/q} < \infty \bigg\}
\end{align}
for some $\alpha \in \mathbb{N}$ and $p < q < \infty$.
This space consists of functions whose weak derivatives up to order $\alpha$ are in $L_\rho^q$. Since every element of this space admits a continuous representative, in what follows we always consider the continuous representative so that the function evaluation makes sense.

Function approximation in high dimension is an important task.
The sampling recovery problem in Gaussian Sobolev spaces has been recently studied by a number of authors~\cite{LPE2023,DK2023,GHRR2023}.
In this paper, we study sampling recovery in the one-dimensional Gaussian Sobolev spaces $W_\rho^{\alpha,q}$ and construct explicit linear algorithms that attain the optimal or almost optimal rate of worst-case convergence.
It has been shown by Kuo, Plaskota, and Wasilkowski \cite[Theorem~3]{KPW2016}, and later independently studied by D{\~u}ng and Nguyen~\cite[Theorem~3.3]{DK2023}, that the best possible convergence rate of sampling recovery in $W_\rho^{\alpha,q}$ for $p<q$ is of order $n^{-\alpha}$, and we show two alternative methods achieving this rate with explicit computational cost. 
As one-dimensional algorithms often serve as a foundation for their high-dimensional counterparts, our results motivate further study of such optimal algorithms in a high-dimensional setting.

The first algorithm we propose in Section~\ref{sec:trig-interpolation}, scaled trigonometric interpolation, achieves the optimal convergence rate up to a logarithmic factor. This algorithm is nothing but a trigonometric interpolation on a suitably truncated interval, and thus can be constructed by Fast Fourier Transform (FFT) with $\calO(n\log_2 n)$ computational cost and $\calO(n)$ memory usage. Trigonometric interpolation is a popular numerical tool due to the applicability of FFT. In spectral methods, trigonometric bases are often used even if the original partial differential equation is defined on the real line (e.g., \cite{K2007,TCN2009}).
This requires suitable treatment of boundary conditions and truncation of the real line.
The truncation interval is chosen to be \emph{large enough}, but analysis of the error that this treatment causes is often omitted.
Our results on trigonometric interpolation give theoretical foundation for such methods, and suggest how to choose the truncation interval depending on the number of evaluation points.

In Section~\ref{sec:optimal} we propose the second algorithm, which is based on spline smoothing.
The algorithm converges with the optimal rate $n^{-\alpha}$ and is similar to the numerical integration algorithms in~\cite{Karvonen2021,DK2023}.
This is achieved by partitioning a truncated real line into unit intervals and constructing an independent spline smoother on each interval.
As spline smoothers are known to attain the optimal rate of convergence in classical Sobolev spaces on bounded sets, having the number of evaluation points allocated to an interval decrease exponentially fast as a function of the distance of the interval to the origin ensures that the algorithm attains the optimal rate.
The slightly improved convergence rate of this algorithm in comparison to scaled trigonometric interpolation is offset by its higher computational complexity.
Figures~\ref{fig:points} and~\ref{fig:interpolation} compare the two algorithms that we propose.
 
Numerical integration is closely related to function approximation. Here we mention results related to numerical integration over Gaussian Sobolev spaces for $q=2$. 
For the deterministic worst-case error, it has been shown that the rate $n^{-\alpha}$ is optimal by Dick et al. \cite[Theorem~1]{DILP2018}, meaning that any deterministic linear quadrature cannot achieve a convergence rate faster than $n^{-\alpha}$. We note that their result is shown for so-called Hermite spaces, but the norm equivalence between Hermite spaces and Gaussian Sobolev spaces with $q=2$ is shown in \cite[Lemma~6]{DILP2018} and \cite[Lemma~2.1]{KSG2022}. Recently, D{\~u}ng and Nguyen \cite[Theorem~2.3]{DK2023} showed that this rate is unchanged for general $q\in(1,\infty)$.
The measure being Gaussian, Gauss--Hermite quadrature is a natural choice.
However, Kazashi, Goda and one of the present authors~\cite{KSG2022} have shown that Gauss--Hermite quadrature achieves the deterministic worst-case rate of only $n^{-\alpha/2}$. An upper bound of order $n^{-\alpha/2}$ was essentially obtained already by Mastroianni and Monegato \cite{MM1994} and the matching lower bound was proved in \cite[Theorem~3.2]{KSG2022}. In contrast, a suitably truncated trapezoidal rule is shown to achieve the optimal rate $n^{-\alpha}$ up to a logarithmic factor in \cite[Theorem~4.5]{KSG2022}.
The subsequent work of these authors \cite{GKS2023} considers the randomized setting, i.e., algorithms are allowed to be random and their quality is measured in the worst-case root-mean-squared error (RMSE).
Therein it is proved that no nonlinear adaptive algorithm can converge with a worst-case RMSE rate faster than $n^{-\alpha-1/2}$, and a randomized trapezoidal rule is shown to attain this rate.
The strategy used for obtaining upper bounds of trapezoidal rules both in \cite{KSG2022,GKS2023}
is based on introducing an auxiliary periodic function. This strategy is first used in \cite{NS2023}, and we adapt this strategy for function approximation in Section~\ref{sec:trig-interpolation}.

Upon completing this research in the paper, we learned the paper \cite{D2025} by D\~{u}ng. The methods presented therein achieve the optimal rate of convergence, however they depend  on various indices. We emphasize that our methods presented in this paper offer computationally cheap alternatives, especially the trigonometric interpolation. As an additional feature, our methods can be constructed without knowledge of the exact smoothness of the target function. Section~\ref{sec:opt} concludes this paper by briefly discussing known lower bounds, optimality, and other features of our algorithms.

\paragraph{Notation.}
Throughout the paper, we denote the set of all positive integers by $\N$.
Unless otherwise specified, in what follows we always assume that $p$ and $q$ in~\eqref{eq:error-measure} and~\eqref{eq:gaussian-sobolev-space} satisfy $1\le p <q < \infty$. Over a given interval $[a,b]$, we denote unweighted and weighted $L^p$ spaces by $L^p([a,b])$ and $L_{\rho}^p([a,b])$, respectively, and likewise for unweighted and weighted Sobolev spaces. 
We use subscripts to denote the dependency of non-negative constants on various parameters. For example, $C_{\alpha, \beta}$ would be a constant that depends \emph{only} on $\alpha$ and $\beta$. Otherwise we always specify the parameters that a constant depends or does not depend on.

\begin{figure}[t]
    \centering
    \includegraphics[width=\textwidth]{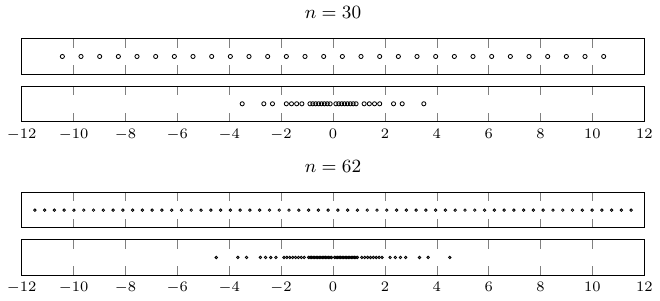}
    \caption{Comparison of the evaluation points used by the algorithms $\Aa_n$ and $\Ab_n$ from Sections~\ref{sec:trig-interpolation} and~\ref{sec:optimal}, respectively, for $n = 30$ and $n = 62$. The uniform points used by $\Aa_n$ in~\eqref{eq:Aa-def}, with $T$ given in~\eqref{eq:T-selection} for $\alpha = 4$, $q=2$, $p=1$ and $\varepsilon = 0.25$, are displayed in the top panels. The bottom panels display one possible set of points that can be used to construct the algorithm $\Ab_n$ of Corollary~\ref{cor:spline-convergence}. The density of these points decreases exponentially fast when moving away from the origin. The algorithm $\Ab_n$ converges with the optimal rate $n^{-\alpha}$, while the rate of convergence of $\Aa_n$ is optimal up to a logarithmic factor.}
    \label{fig:points}
\end{figure}

\begin{figure}[t]
    \includegraphics[width=0.49\textwidth]{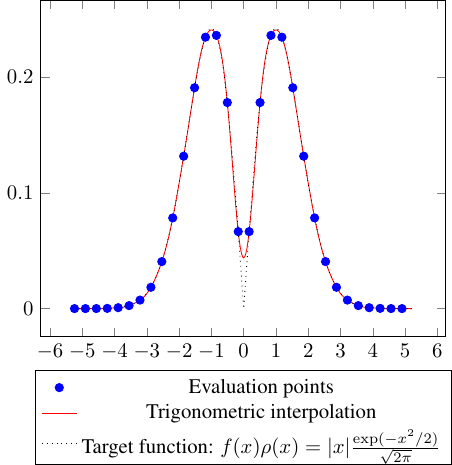}
    \includegraphics[width=0.49\textwidth]{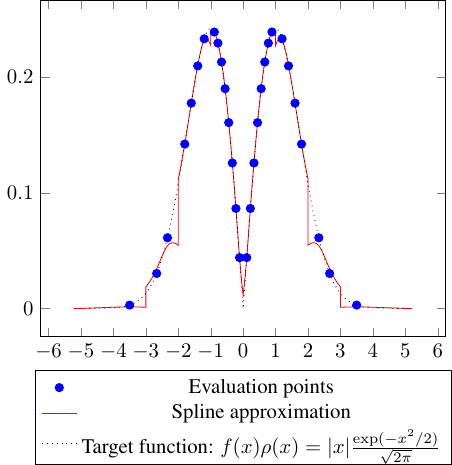}
    \caption{A comparison of trigonometric interpolation from Section~\ref{sec:trig-interpolation} and spline approximation from Section~\ref{sec:optimal} when $f(x) = \lvert x \rvert$, $p=1$ and $q=2$. We plot the weighted approximations $\rho \Aa_n(f)$ with $n=31$ and $\varepsilon=0.25$ (left) and $\rho \Ab_n(f)$ with $n=30$ (right). The spline approximation has been constructed using a Matérn kernel of order $\gamma = 3/2$ [see~\eqref{eq:matern-kernel}], which is a reproducing kernel for a Sobolev space of order two. Observe how the spline approximation is more accurate close to the origin than the trigonometric interpolant.}
    \label{fig:interpolation}
\end{figure}
 
\section{Trigonometric interpolation} \label{sec:trig-interpolation}
This section introduces our first algorithm based on trigonometric interpolation and analyses its $L^p_{\rho}$-error in terms of the number $n$ of function evaluation. 
The algorithm is denoted by $\Aa_n:W^{\alpha,q}_\rho \to L^{p}_\rho$. 
\begin{definition}[Trigonometric interpolation with cutoff $\Aa_n$]
Let $g \coloneqq f\rho^{1/p}$. We define the approximation algorithm 
\begin{equation} \label{eq:Aa-def}
    \Aa_n (f) \coloneqq \rho^{-1/p}\calI_n(f\rho^{1/p})=\rho^{-1/p}\calI_n(g),
\end{equation}
with $\calI_n$ being the trigonometric interpolation on $[-T,T]$:
\begin{align*}
\calI_n(g)
\coloneqq
 \begin{cases}
   \sum_{k=-\lfloor n/2\rfloor}^{\lfloor n/2\rfloor} \widehat{g}_a (k) \phi_k^{[-T,T]}(x)& \text{if $x\in [-T,T],$}\\ 0 & \text{otherwise,}
 \end{cases}
\end{align*}
where $\phi_k^{[-T,T]}(x) \coloneqq \exp(2\pi i k(x+T)/(2T))/\sqrt{2T}$ are orthonormal Fourier basis on $L^2([-T,T])$.
The coefficients $\widehat{g}_a(k)$ are calculated using equidistant points as
\[
\widehat{g}_a (k) \coloneqq \frac{1}{n}\sum_{j=0}^{n-1} g(\xi_j) \overline{\phi_k^{[-T,T]}}(\xi_j), \quad \text{ with } \quad \xi_j \coloneqq \frac{2T}{n}j-T,\ j=0,\dots,n-1.
\]
\label{def:An}
\end{definition}

We will bound the $L^p_\rho$ error of the above algorithm by decomposing it into two parts:
\begin{equation}
\begin{split}
\|f-\Aa_n(f)\|_{L_\rho^p} \le{}& 
\left(\int_{\R\setminus [-T,T]} \left|f(x)-[\Aa_n(f)](x)\right|^p \rho(x) \rd x  \right)^{1/p}  \\
&+
\left(\int_{-T}^{T} \left|f(x)-[\Aa_n(f)](x)\right|^p \rho(x) \rd x  \right)^{1/p}
.
\label{eq:error-decomp}
\end{split}
\end{equation}
For the first term, since we have $\Aa_n(f)=0$ outside the interval $[-T,T]$, we simply bound $(\int_{\R\setminus [-T,T]} |f(x)|^p \rho(x) \rd x )^{1/p} $ using the decay of $g=f\rho^{1/p}$. In the following lemma, we obtain this decay of $g$ and its derivatives.

\begin{lemma}[Decay of the function $g=f\rho^{1/p}$]\label{lem:decay} 
    Let $1\le p <q  < \infty$, $\alpha\in\N$ and $f\in W^{\alpha,q}_{\rho}$.
    Then for arbitrary $\varepsilon\in(0,\frac{q-p}{pq})$ and $g=f\rho^{1/p}$, the following quantity is bounded: 
        \[    
    \|g\|_{\textup{decay}} \coloneqq
    \sup _{\substack{x \in \mathbb{R} \\ \tau \in\{0, \ldots, \alpha-1\}}}\big|\rho^{1/q-1/p+\varepsilon}(x) g^{(\tau)}(x)\big|\le C_{p,q,\alpha,\varepsilon} \|f\|_{W^{\alpha,q}_{\rho}}<\infty .
    \]
\end{lemma}
\begin{proof} First we note that $1/q-1/p+\varepsilon<0$.
        In the following, we use the Sobolev inequality for $F_\tau(x) \coloneqq \rho^{1/q-1/p+\varepsilon}(x) g^{(\tau)}(x)$. That is, from the boundedness of $\|F_\tau\|_{L^q(\R)}$ and $\|F_\tau '\|_{L^q(\R)}$ we deduce the boundedness of $\|F_\tau \|_{L^\infty(\R)}$.
        Let $H_\ell(x)$ be the $\ell$-th degree probabilist's Hermite polynomial,
        \[
        H_\ell(x)=\frac{(-1)^\ell}{\sqrt{\ell!}}\e^{x^2/2 } \frac{\rd^{\ell}}{\rd x^{\ell}}
        	\e^{-x^{2}/2}.
       \]
        First notice from the definition of Hermite polynomials that
   \[
   \frac{\mathrm{d}^\ell}{\mathrm{d} x^\ell} \rho(x)^{1/p}=p^{-\ell/2} (-1)^\ell {\sqrt{\ell !}} H_\ell\left(\frac{x}{\sqrt{p}}\right)  \rho(x)^{1/p}.
   \] 
Hence, for $\tau=0,\ldots,\alpha-1,$ by applying the product rule to $g^{(\tau)}$ we have 
\begin{align*}
    &\|F_\tau(x)\|_{L^q(\R)}\\
    &\le
    \sum_{\ell=0}^\tau \binom{\tau}{\ell} \|
    p^{-\ell/2} (-1)^\ell {\sqrt{\ell !}} H_\ell\left(\frac{x}{\sqrt{p}}\right)  \rho(x)^{1/p}
    \rho^{1/q-1/p+\varepsilon}(x) f^{(\tau-\ell)}(x)\|_{L^q(\R)}
    \\
    & \le
    \sum_{\ell=0}^\tau \binom{\tau}{\ell} {\sqrt{\ell !}}p^{-\ell/2}  \|
       H_\ell\left(\frac{x}{\sqrt{p}}\right)  \rho(x)^{1/p}
    \rho^{1/q-1/p+\varepsilon}(x) f^{(\tau-\ell)}(x)\|_{L^q(\R)}
    \\
    &\le
    \sum_{\ell=0}^\tau \binom{\tau}{\ell} {\sqrt{\ell !}}p^{-\ell/2}
      \sup_{x\in\R} \left| H_\ell\left(\frac{x}{\sqrt{p}} \right)\rho^{\varepsilon}(x)    \right|
    \left(\int_\R |f^{(\tau-\ell)}(x)|^q \rho(x) \rd x\right)^{1/q}
    \\
    &=
    \sum_{\ell=0}^\tau \binom{\tau}{\ell} {\sqrt{\ell !}}p^{-\ell/2}
    \sup_{x\in\R} \left| H_\ell\left(\frac{x}{\sqrt{p}} \right)\rho^{\varepsilon}(x)    \right|
    \|f^{(\tau-\ell)}(x)\|_{L^{q} _{\rho}}
    <\infty,
\end{align*}
where in the last line, $\|f^{(\tau-\ell)}(x)\|_{L^{q} _{\rho}}$ is bounded because $f\in W^{\alpha,q}_{\rho}$, and the supremum term is also bounded because $H_\ell$ has at most only polynomial growth.
At the same time, noting that $1/q-1/p+\varepsilon<0$, we have
\begin{align*}
&\|F_{\tau} '\|_{L^q(\R)} \\\le{}&
\|\rho^{1/q-1/p+\varepsilon}(x) g^{(\tau+1)}\|_{L^q(\R)}
+
\|(1/p-1/q-\varepsilon)x\rho^{1/q-1/p+\varepsilon}(x) g^{(\tau)}\|_{L^q(\R)}
\\
\le{}&
\sum_{\ell=0}^{\tau+1} \binom{\tau+1}{\ell} {\sqrt{\ell !}}p^{-\ell/2}
 \sup_{x\in\R} \left| H_\ell\left(\frac{x}{\sqrt{p}} \right)\rho^{\varepsilon}(x)    \right|
    \|f^{(\tau-\ell+1)}(x)\|_{L^{q} _{\rho}}\\
    &+
    (1/p-1/q-\varepsilon)\sum_{\ell=0}^\tau \binom{\tau}{\ell} {\sqrt{\ell !}}p^{-\ell/2}
    \sup_{x\in\R} \left|x H_\ell\left(\frac{x}{\sqrt{p}} \right)\rho^{\varepsilon}(x)    \right|
    \|f^{(\tau-\ell)}(x)\|_{L^{q} _{\rho}}
    < \infty,
\end{align*}
where in the last line, again we used $ \|f^{(\tau-\ell+1)}(x)\|_{L^{q} _{\rho}}\le\|f\|_{W^{\alpha,q}_{\rho}}<\infty$.
Thus we have $\|g\|_{\textup{decay}}<\infty$.
\qed
\end{proof}
For the second term in \eqref{eq:error-decomp}, we introduce a suitable auxiliary function $G$ and consider the bound
\begin{align}
    &\bigg(\int_{-T}^{T} |f(x)-[\Aa_n(f)](x)|^p \rho(x)  \rd x  \bigg)^{1/p}\nonumber \\={}& \bigg(\int_{-T}^{T} |g(x)-[\calI_n(g)](x)|^p   \rd x  \bigg)^{1/p} \nonumber \\
    ={}& \|g-\calI_n(g)\|_{L^p([-T,T])} \nonumber \\
    \begin{split}
    \leq{}& \|g-G\|_{L^p([-T,T])} + \|G-\calI_n(G)\|_{L^p([-T,T])} + \|\calI_n(G)-\calI_n(g)\|_{L^p([-T,T])}. \label{eq:error-dec-in}
    \end{split}
\end{align}
Lemma~\ref{lem:error-inside} will upper bound the three terms in~\eqref{eq:error-dec-in}. For that, we need the following Lemmas~\ref{lem:periodize} to~\ref{lem:interpolation-box}.

\begin{lemma}[Auxiliary periodic function and its properties]\label{lem:periodize}
Let $1\le p <q  < \infty$, $\alpha\in\N$ and $T > 0$.  For $f\in W^{\alpha,q}_{\rho} $, let $g=f \rho^{1/p}$ and define $G: [-T-d, T+d]\to \R$ by
\begin{align}\label{eq:periodize}
    G(x) \coloneqq g(x)-\sum_{\tau=1}^{\alpha}\frac{B^{[-T,T]}_{\tau}(x)}{\tau!}\left( \int_{-T}^{T}g^{(\tau)}(y)\rd y\right),
\end{align}
 with an arbitrarily small $d>0$, where $B^{[-T,T]}_{\tau}$ denotes the scaled Bernoulli polynomial of degree $\tau$ on $[-T,T]$, i.e.,
\[ B^{[-T,T]}_{\tau}(x)=(2T)^{\tau-1}B_{\tau}\left( \frac{x+T}{2T}\right)\]
with $B_{\tau}$ being the standard Bernoulli polynomial of degree $\tau$. Then, the auxiliary function $G$ is $(\alpha-1)$-times continuously differentiable with $G^{(\alpha-1)}$ being absolutely continuous on $[-T,T]$, and satisfies $G^{(\tau)}(-T)=G^{(\tau)}(T)$ for all $\tau=0,\ldots,\alpha-1$. Furthermore, the auxiliary function $G$ has its $\alpha$-th weak derivative $G^{(\alpha)}$ in $L^q([-T,T])$ which can be identified as a function over a $2T$-periodic torus $\T([-T,T))$ with
\[
\int_{\T([-T,T))} |G^{(\alpha)} (x)|^q \rd x < \infty.
\]
\end{lemma}
\begin{proof}
   The proof follows the strategy of \cite[Lemma~4.1]{KSG2022}, where the case with $p=1$ and $q=2$ was proved. 
   Recall that for $\tau,\tau' \in \N_0$ scaled Bernoulli polynomials satisfy 
   \[
   \frac{\mathrm{d}^{\tau'}}{\mathrm{d} x^{\tau'}} \frac{B^{[-T,T]}_\tau(x)}{\tau !} =\left\{\begin{array}{ll}
0 & \text { if } \tau^{\prime}>\tau, \\
\frac{B^{[-T,T]}_{\tau-\tau'}(x)}{\left(\tau-\tau'\right) !} & \text { if } \tau' \leq \tau
\end{array}\right.
   \]
   and
   \[
\int_{-T} ^T B_\tau^{[-T, T]}(x) \rd x=\left\{\begin{array}{ll}
1 & \text { if } \tau=0, \\
0 & \text { if } \tau \neq 0.
\end{array}\right.
   \]
   From these properties we have
\begin{align*}
&\int_{-T}^{T} G^{(\tau)}(x)\rd x =\int_{-T}^{T} g^{(\tau)}(x)\rd x - \left(\int_{-T}^{T} B^{[-T,T]}_{0}(x) \rd x \right)\left(\int_{-T}^{T} g^{(\tau)}(s)\rd s\right)  =0 ,
\end{align*}
for $\tau=1,\ldots,\alpha$. Since $f\in W^{\alpha,q}_{\rho}$, $g^{(\tau)}$ and $G^{(\tau)}$ are absolutely continuous on $(-T-d,T+d)$ for $\tau = 0,1,\ldots,\alpha-1$. By using the fundamental theorem of calculus, we now obtain
\begin{align}\label{eq:boundary-match}
&G^ {(\tau)}(-T)=G^ {(\tau)}(T)\quad \text{for }\  \tau=0,\ldots,\alpha-1.
\end{align}
   Next, we note that $G^{(\alpha)}$ and $g^{(\alpha)}$ only differ by the constant $\frac{1}{2T} \int_{-T}^T g^{(\alpha)}(y) \rd y$ on $[-T,T]$. Since $f^{(\alpha)}\in L^{q}_{\rho}$, we know $g^{(\alpha)}$ is in $L^q([-T,T])$.
   Due to matching boundary values $G^{(\alpha-1)}(-T)=G^{(\alpha-1)}(T)$ and absolute continuity of $G^{(\alpha-1)}$ on $(-T-\delta,T+\delta)$, 
   we have $\|G^{(\alpha)}\|_{L^p(\T([-T,T]))} < \infty $.
\qed \end{proof}

\begin{lemma}[Norm estimate on bounded intervals] \label{lem:g-norm-bound}
    Let $1\le p <q  < \infty$, $\alpha\in\N$, $a < b$ and $\varepsilon \in [0, q - p)$.
    Then
    \[
    \|g\|_{W^{\alpha,q} ([a, b])} \le C_{p,q,\alpha,\varepsilon} \, \rho^{\varepsilon/p}(c) \|f\|_{W^{\alpha,q}_{\rho}},
    \]
    where $c = \min_{x \in [a, b]} \lvert x \rvert$.
\end{lemma}
\begin{proof}
We have
\begin{align*}
 &\|g\|_{W^{\alpha,q} ([a, b])}
 = \left( \sum_{\tau=0}^{\alpha} \int_{a}^b \lvert  g^{(\tau)} (x) \rvert^q \rd x \right)^{1/q}
 =  \left(\sum_{\tau=0}^{\alpha} \int_{a}^b \lvert  (f\rho^{1/p}))^{(\tau)} (x) \rvert^q \rd x \right)^{1/q}
\\
 &\le  \sum_{\tau=0}^{\alpha} \left( \int_{a}^b \bigg\lvert   \sum_{\ell=0}^\tau \binom{\tau}{\ell} (-1)^\ell {\sqrt{\ell !}}f^{(\tau-\ell)}(x) p^{-\ell/2} H_\ell\bigg(\frac{x}{\sqrt{p}}\bigg)\rho^{1/p}(x)  \bigg\rvert^q \rd x \right)^{1/q}
 \\
 &\le
 \sum_{\tau=0}^{\alpha} \sum_{\ell=0}^\tau  \binom{\tau}{\ell}  \left( \int_{a}^b \bigg\lvert f^{(\tau-\ell)}(x){\sqrt{\ell !}} p^{-\ell/2} H_\ell\bigg(\frac{x}{\sqrt{p}}\bigg)\rho^{1/p}(x) \bigg\rvert^q \rd x \right)^{1/q}
   \\
 &\le
 \sum_{\tau=0}^{\alpha} \sum_{\ell=0}^\tau  \binom{\tau}{\ell}  \left( \sup_{x\in[a,b]}\bigg\lvert {\sqrt{\ell !}} p^{-\ell/2} H_\ell\bigg(\frac{x}{\sqrt{p}} \bigg)\rho^{(q-p)/p}(x) \bigg\rvert^q  \int_{a}^b \lvert f^{(\tau-\ell)}(x) \rvert^q \rho(x) \rd x \right)^{1/q}
    \\
 &\leq \sum_{\tau=0}^{\alpha} \sum_{\ell=0}^\tau  \binom{\tau}{\ell}  \bigg( \sup_{x\in[a,b]}\bigg\lvert {\sqrt{\ell !}}p^{-\ell/2} H_\ell\bigg(\frac{x}{\sqrt{p}} \bigg)\rho^{(q-p-\varepsilon)/p}(x)  \bigg\rvert^q\\ &\quad \quad \kern35mm \times \rho^{\varepsilon q/p}(c)  \int_{a}^b \lvert f^{(\tau-\ell)}(x) \rvert^q \rho(x) \rd x \bigg)^{1/q}
 \\
 &\le
 C_{p,q,\alpha,\varepsilon} \, \rho^{\varepsilon/p}(c) \|f\|_{W^{\alpha,q}_{\rho}},
\end{align*}
where we used the facts that $H_\ell(\cdot/\sqrt{p}) \rho^{(q-p-\varepsilon)/p}$ is bounded on the real line and that $\rho(x)$ is decreasing in $\lvert x \rvert$.
\qed \end{proof}
For obtaining trigonometric interpolation error on $[-T,T]$, we make use of results from trigonometric interpolation on $[0,2\pi]$. For that sake, we need the following estimate for the effect of linear scaling $S_T:[0,2\pi]\to [-T,T]$ given by $S_T(x) \coloneqq Tx/\pi-T$.

\begin{lemma}[Boundedness of $\|g\circ S_T\|_{W^{\alpha,q}([0,2\pi])}$] \label{lem:scaling-norm}
    Let $1\le p <q < \infty$, $\alpha\in\N$, and $f\in W^{\alpha,q}_{\rho}$.
    For $g=f\rho^{1/p}$ we define $h \coloneqq g\circ S_T$. 
    Then, for $T\ge1$ we have
    \[
    \|h\|_{W^{\alpha,q} ([0,2\pi])}\le C_{p,q,\alpha} T^{\alpha-1/q} \|f\|_{W^{\alpha,q}_{\rho}}.
    \]
\end{lemma}
\begin{proof}
We have
\begin{align*}
 \|h\|_{W^{\alpha,q} ([0,2\pi])}
 &=
 \left(\sum_{\tau=0}^{\alpha} \int_0^{2\pi} \lvert h^{(\tau)} (x) \rvert^q \rd x\right)^{1/q}
 \\&=
 \left(\sum_{\tau=0}^{\alpha} \int_0^{2\pi} \lvert (g \circ S_T )^{(\tau)} (x) \rvert^q \rd x \right)^{1/q}
 \\
 &=
 \left( \sum_{\tau=0}^{\alpha} \int_0^{2\pi} \lvert (T/\pi)^{\tau} (g^{(\tau)} \circ S_T ) (x) \rvert^q \rd x\right)^{1/q}
 \\
 &= \left( \sum_{\tau=0}^{\alpha} (T/\pi)^{\tau q - 1} \int_{-T}^T \lvert  g^{(\tau)} (x) \rvert^q \rd x \right)^{1/q}
 \\
 &\le
 C_{q,\alpha} T^{\alpha - 1/q} \lVert g \rVert_{W^{\alpha,q}([-T, T])}.
\end{align*}
The claim then follows from Lemma~\ref{lem:g-norm-bound} with $a = -T$, $b = T$, and $\varepsilon = 0$.
\qed \end{proof}

\begin{lemma}[$L^p$ interpolation error for periodic functions]\label{lem:interpolation-box}
Let $1\le p <q < \infty$, $\alpha\in\N$, and $f\in W^{\alpha,q}_{\rho}$.
Define $g=f\rho^{1/p}$ and an auxiliary periodic function $G$ as in \eqref{eq:periodize}. 
Then
\[
\|G-\calI_n(G)\|_{L^p ([-T,T])}\le C_{p,q,\alpha} T^{\alpha-1/q+1/p} \|f\|_{W^{\alpha,q}_{\rho}} \frac{1}{n^{\alpha}}.
\]

\end{lemma}
\begin{proof}
The result by Temlyakov \cite[Theorem~2.4.4]{T2018_book}, where the interpolation operator is denoted by $I_m$, applies to our setting. Therein, the function space considered is defined via convolution kernels, but the norm equivalence to periodic Sobolev spaces 
\[
W^{\alpha,q}(\T[0,2\pi)) \coloneqq \big\{f \; \big| \; f\in W^{\alpha,q}([0,2\pi]), f^{(\tau)}(0)=f^{(\tau)}(2\pi) \;\text{for} \;\tau=0,\ldots,\alpha-1 \big\}
\]
is shown in \cite[Theorem~2.7]{U2007}. Hence, by letting $S_T(x) \coloneqq Tx/\pi-T$, \cite[Theorem~2.4.4]{T2018_book} leads to
\[
\|G\circ S_T-\calI_n(G \circ S_T)\|_{W^{\alpha,q}(\T(0,2\pi))} \le C'_{p,q,\alpha} \|G\circ S_T \|_{W^{\alpha,q}(\T(0,2\pi))}\frac{1}{n^\alpha}.
\]
Further, by Lemma~\ref{lem:scaling-norm}, the above error is bounded by 
\[
\le  C''_{p,q,\alpha} T^{\alpha-1/q} \|f\|_{W^{\alpha,q}_{\rho}}\frac{1}{n^\alpha}.
\]
At the same time, since $S_T$ is a plain linear scaling, we have
\[
\|G\circ S_T-\calI_n(G\circ S_T ) \|_{L^p(\T(0,2\pi))}=
\left(\frac{\pi}{T}\right)^{1/p}\|G -\calI_n(G) \|_{L^p(\T(-T,T))}.
\]
Thus we have proved the claim.
\qed \end{proof}

\begin{lemma}[Error bound on {$[-T,T]$} for $\Aa_n$]\label{lem:error-inside}
    Let $1\le p <q < \infty$, $\alpha\in\N$, and $f\in W^{\alpha,q}_{\rho}$. Choose $\varepsilon\in(0,\frac{q-p}{pq})$ arbitrarily. Then the approximation error on $[-T,T]$ by the algorithm $\Aa_n$ is bounded by
    \begin{align*}
    &\|f-\Aa_n(f)\|_{L_\rho ^p([-T,T])}
=
\left(\int_{-T}^T \left| f(x)\rho^{1/p}(x)-\calI_n(f\rho^{1/p}) \right|^p \rd x\right)^{1/p}
 \\
& \le C_1 \|f\|_{W^{\alpha,q}_{\rho}} T^{\alpha-1/q+1/p} n^{-\alpha}\\&\kern20mm+C_2 \|f\|_{W^{\alpha,q}_{\rho}} (\ln n)\max\{1,(2T)^{\alpha-1+1/p}\} \e^{-\frac{(q-p)/pq-\varepsilon}{2}T^2},
    \end{align*}
    where the constants $C_1$ and $C_2$ do not depend on $n$ or $T$.
\end{lemma}
\begin{proof}
    From Lemma~\ref{lem:decay}, we know that for $g \coloneqq f\rho^{1/p}$ the quantity
    \[
    \|g\|_{\text{decay}} =
    \sup _{\substack{x \in \mathbb{R} \\ \tau \in\{0, \ldots, \alpha-1\}}}\left|\rho^{1/q-1/p+\varepsilon}(x) g^{(\tau)}(x)\right|
    \]
    is bounded. 
As mentioned above, we consider the following bound:
\begin{align}
     &\|f-\Aa_n(f)\|_{L_\rho ^p([-T,T])} 
     =\|g-\calI_n{g}\|_{L^p ([-T,T])} \nonumber\\
     & \le \|g-G\|_{L^p ([-T,T])} + \|G-\calI_n(G)\|_{L^p ([-T,T])} + \|\calI_n(G)-\calI_n(g)\|_{L^p ([-T,T])}. \label{eq:upper-three-terms}
\end{align}

For the first term in the last line of \eqref{eq:upper-three-terms}, we have
\begin{align}\label{eq:err-term1}
\|g-G\|_{L^p ([-T,T])}
&\le
\sum_{\tau=1}^{\alpha}\frac{\|B^{[-T,T]}_{\tau}(x)\|_{L^p ([-T,T])}}{\tau!}\left|\int_{-T}^T g^{(\tau)}(s) \rd s\right|\nonumber
\\
&\le
\sum_{\tau=1}^{\alpha}\frac{\|B^{[-T,T]}_{\tau}(x)\|_{L^\infty ([-T,T])} \; (2T)^{1/p}}{\tau!}\left|\int_{-T}^T g^{(\tau)}(s) \rd s\right|\nonumber
\\
&\le
\sum_{\tau=1}^{\alpha}(2T)^{\tau-1+1/p}\left|\int_{-T}^T g^{(\tau)}(s) \rd s\right|\nonumber
\\
&\le
\alpha \max\{1,(2T)^{\alpha-1+1/p}\}\|g\|_{\text{decay}}\;   C_{p,q,\varepsilon}\; \e^{-\frac{(q-p)/pq-\varepsilon}{2}T^2},
\end{align}
where in the penultimate line, we used 
$|\frac{B^{[-T,T]}_{\tau}(x)}{\tau!}|\le \frac{(2T)^{\tau-1}}{2}$ for $x\in[-T,T]$; see~\cite[Equation~(6)]{NS2023} or~\cite{L1940}.

We proceed to the third term. We have
\begin{align}\label{eq:err-term2}
&\|\calI_n(g-G)\|_{L^p ([-T,T])} \nonumber \\
&\le
\|\calI_n\|_{C([-T,T])\to L^p ([-T,T])} \|g-G\|_{L^p ([-T,T])}\nonumber
\\
&\le
\|\calI_n\|_{C([-T,T])\to L^p ([-T,T])}\alpha \max\{1,(2T)^{\alpha-1+1/p}\} \;\|g\|_{\text{decay}}\;  C_{p,q,\varepsilon}\; \e^{-\frac{(q-p)/pq-\varepsilon}{2}T^2}.
\end{align}
Now the operator norm of the interpolation, from continuous functions to $L^p$, can be bounded using the Lebesgue constant $\Lambda_n$ for trigonometric interpolation by
\[
\|\calI_n\|_{C([-T,T])\to L^p ([-T,T])} \le 2T \Lambda_n \le 2T \left(\frac{\pi+4}{\pi}+\frac{2\ln n}{\pi}\right),
\] where this explicit constant here can be found, e.g., in \cite[Theorem~2.1]{SN2016}, or we refer to \cite{EZ1966} as an earlier result. 

For the second term in \eqref{eq:upper-three-terms}, we know from Lemma~\ref{lem:interpolation-box} that
\begin{align}\label{eq:err-term3}
 \|G-\calI_n(G)\|_{L^p ([-T,T])} \le  C_{p,q,\alpha} T^{\alpha-1/q+1/p} \|f\|_{W^{\alpha,q}_{\rho}} \frac{1}{n^{\alpha}}
 .
\end{align}
Summing up all three terms, \eqref{eq:err-term1},\eqref{eq:err-term2}, and \eqref{eq:err-term3}, we obtain the error bound
\begin{align*}
       & \|f-\Aa_n(f)\|_{L_\rho ^p([-T,T])}
 \\&\le C_1 T^{\alpha-1/q+1/p} n^{-\alpha}+C_2 (\ln n)\max\{1,(2T)^{\alpha-1+1/p}\} \e^{-\frac{(q-p)/pq-\varepsilon}{2}T^2}.
\end{align*}
Hence we proved the claim.\qed 
\end{proof}

\begin{lemma}[Error bound on tails] \label{lemma:tail-bound}
Let $1\le p <q < \infty$, $\alpha\in\N$, and $f\in W^{\alpha,q}_{\rho}$. Set $g=f\rho^{1/p}$. Choose $\varepsilon\in(0,\frac{q-p}{pq})$ arbitrarily. For the error outside the interval $[-T,T]$, we have
\begin{align*}
\left(\int_{\R\setminus [-T,T]} \left|f(x)\right|^p \rho(x) \rd x\right)^{1/p}
&\le
  C_{p,q,\varepsilon}\;\|g\|_{\textup{decay}}\; T^{-1/p} \; \e^{-\frac{(q-p)/pq-\varepsilon}{2}T^2}.
\end{align*}
\end{lemma}
\begin{proof}
Let $r \coloneqq 1/p-1/q-\varepsilon > 0$. Using Lemma~\ref{lem:decay} we have
\begin{align*}
&\left(\int_{\R\setminus [-T,T]} \left|g(x)\right|^p \rd x\right)^{1/p}
\le
\left(\int_{\R\setminus [-T,T]} \left(\|g\|_{\text{decay}} \frac{\e^{-rx^2/2}}{\sqrt{2\pi}^r} \right)^p \rd x\right)^{1/p} 
\\
&\le
\|g\|_{\text{decay}} \left(\frac{2}{T}\int_{T}^\infty  \frac{x\e^{-rpx^2/2}}{\sqrt{2\pi}^{rp}} \rd x\right)^{1/p} 
\\
&=
\|g\|_{\text{decay}}   \left(\frac{2}{rpT\sqrt{2\pi}^{rp}} \; \e^{-rpT^2/2}\right)^{1/p} 
=
\frac{2^{1/p}}{(rpT)^{1/p}\sqrt{2\pi}^{r}} \; \e^{-rT^2/2}.
\end{align*}
This proves the claim.\qed \end{proof}

\begin{theorem}[Error bound on {$\R$} for $\Aa_n$] \label{thm:err-total}
    Let $1\le p <q < \infty$, $\alpha\in\N$, and $f\in W^{\alpha,q}_{\rho}$. Choose $\varepsilon\in(0,\frac{q-p}{pq})$ arbitrarily. 
    Set the cut-off interval $[-T,T]$ with
    \begin{align} \label{eq:T-selection}
    T=\sqrt{2\alpha \left(\frac{ q-p}{pq} - \varepsilon \right)^{-1} \ln n}.
    \end{align}
    Then, for any integer $n\ge 2$, the approximation error on $\R$ by the algorithm $\Aa_n$ is bounded by
    \begin{align*}
    \|f-\Aa_n(f)\|_{L_\rho ^p}
&=
\left(\int_\R \left| g(x)\rho^{1/p}(x)-\calI_n(g\rho^{1/p}) \right|^p \rd x\right)^{1/p}
\\&\le
C_{p,q,\alpha,\varepsilon} \|f\|_{W^{\alpha,q}_{\rho}}\frac{(\ln n)^{\frac{\alpha}{2}+\frac{1}{2}+\frac{1}{2p}}}{n^\alpha}.
    \end{align*}
\end{theorem}
\begin{proof}
We consider the error bound \eqref{eq:error-decomp}. With our choice of the cutoff interval, Lemma~\ref{lem:error-inside} tells us that
\begin{align*}
&\left(\int_{-T}^{T} \left|f(x)-[\Aa_n(f)](x)\right|^p \rho(x) \rd x  \right)^{1/p}
\\&\kern45mm\le
C_1 (\ln n)^{\frac{\alpha}{2} -\frac{1}{2q}+\frac{1}{2p}} n^{-\alpha}+C_2 (\ln n)^{\frac{\alpha}{2}+\frac{1}{2}+\frac{1}{2p}} n^{-\alpha},
\end{align*}
here we used the fact that $T$ is always bigger than $1$ for $n\ge 2$, since $0<\frac{q-p}{pq}-\varepsilon<1$.
For the error outside the interval, using Lemma~\ref{lemma:tail-bound}, we obtain
\begin{align*}
&\left(\int_{\R\setminus [-T,T]} \left|f(x)\right|^p \rho(x) \rd x\right)^{1/p}
\\& \kern30mm\le
  C_{p,q,\varepsilon}\;\|g\|_{\text{decay}}\; \e^{-\frac{(q-p)/pq-\varepsilon}{2}T^2}
=
C_{p,q,\varepsilon}\;\|g\|_{\text{decay}}\; n^{-\alpha}.
\end{align*}
Therefore, the total error is bounded by
\ifpreprint
\begin{equation*}
\|f-\Aa_n(f)\|_{L_\rho ^p} \le C_{p,q,\alpha,\varepsilon} \|f\|_{W^{\alpha,q}_{\rho}} (\ln n)^{\frac{\alpha}{2}+\frac{1}{2}+\frac{1}{2p}} n^{-\alpha}.\qedhere
\end{equation*}
\else
\begin{equation*}
\|f-\Aa_n(f)\|_{L_\rho ^p} \le C_{p,q,\alpha,\varepsilon} \|f\|_{W^{\alpha,q}_{\rho}} (\ln n)^{\frac{\alpha}{2}+\frac{1}{2}+\frac{1}{2p}} n^{-\alpha}.
\end{equation*}
\fi
Thus we have proved the claim.\qed \end{proof}
\begin{remark}[Periodic auxiliary function $G$]
    This proof strategy of introducing the auxiliary function $G$ can be applied to other algorithms or problems. Let us elaborate this part here. Consider an abstract problem of approximating a linear operator or a linear functional $S:\calH_1\to\calH_2$ by a linear algorithm $A_n:\calH_1\to\calH_2$. This is a typical setting in \emph{information-based complexity}, see, e.g., \cite[Section~4.4]{EW2008_book1}. Then, the error has the following upper bound for any $g\in\calH_1$:
    \begin{align}\label{eq:abst_err}
    \|S(g)-A_n(g)\|_{\calH_2}
    &\le
    \|S(g)-S(G)\|_{\calH_2} + \|S(G)-A_n(G)\|_{\calH_2} + \|A_n(G)-A_n(g)\|_{\calH_2}\nonumber
    \\
    &=
    \|S(g-G)\|_{\calH_2} + \|S(G)-A_n(G)\|_{\calH_2} + \|A_n(G-g)\|_{\calH_2}.
    \end{align}
    This strategy is first used in \cite{NS2023} for a multivariate integration problem in which $S(g)=\int_{\R^d} g(\bsx) \rd \bsx$ and $A_n$ is a specific quadrature rule, called a scaled rank-$1$ lattice rule. Therein, the first error term in \eqref{eq:abst_err}, $  \|S(g-G)\|_{\calH_2}$, happens to vanish due to properties of the integration functional. However, as we have seen, this term can be bounded just like the third term $\|A_n(G-g)\|_{\calH_2}$ for general linear algorithms.
\end{remark}

In Theorem~\ref{thm:err-total}, the choice of the cut-off parameter $T$ depends on the smoothness $\alpha$ of the approximand. However, this information may not be available in practice. Following \cite[Corollary~4.4]{KSG2022} and \cite[Remark 3.11]{GKS2023}, one can replace $\alpha$ by a slowly increasing function $\gamma(n)$ such as $\max(\ln(\ln n),0)$, and still achieve the convergence rate $n^{-\alpha}$, up to a factor $(\gamma(n) \ln n)^{\frac{\alpha}{2}+\frac{1}{2}+\frac{1}{2p}}$.

\begin{corollary}[$\alpha$-free construction] \label{cor:alpha-free}
    Let assumptions of Theorem~\ref{thm:err-total} be satisfied. 
     Choose $\varepsilon\in(0,\frac{q-p}{pq})$ arbitrarily. Choose a non-decreasing function $\gamma(n): \N \to\R$ satisfying $\lim_{n\to\infty}\gamma(n)=\infty$, 
    and set the cut-off interval $[-T,T]$ with
    \begin{align*}
    T=\sqrt{2\gamma(n) \left(\frac{ q-p}{pq} - \varepsilon \right)^{-1} \ln n}.
    \end{align*}
    Then, for any integer $n\geq\gamma^{-1}(\alpha) \coloneqq \min\{m\in\N\mid\gamma(m)\geq\alpha\}$, the approximation error on $\R$ by the algorithm $\Aa_n$ is bounded by
    \begin{align*}
    \|f-\Aa_n(f)\|_{L_\rho ^p}
&=
\left(\int_\R \left| g(x)\rho^{1/p}(x)-\calI_n(g\rho^{1/p}) \right|^p \rd x\right)^{1/p}
\\&\kern30mm\le
C_{p,q,\alpha,\varepsilon} \|f\|_{W^{\alpha,q}_{\rho}}\frac{(\gamma(n)\ln n)^{\frac{\alpha}{2}+\frac{1}{2}+\frac{1}{2p}}}{n^\alpha}.
    \end{align*}
\end{corollary}

\begin{remark}[Computational cost]
As mentioned in Introduction, the computational cost and memory usage of an FFT implementation of algorithm $\Aa_n$ are only $\calO(n\log_2 n)$ and $\calO(n)$, respectively.
The advantage of the method is mainly this cheap cost of construction.  However, when the evaluation of the approximand function $f$ is much more expensive, one may wish to have the optimal convergence rate including the logarithmic factor, at the sacrifice of cheap construction. In such a case, we propose the optimal algorithm described in Section~\ref{sec:optimal}.
\end{remark}

\begin{remark}[Higher dimensions]
For multidimensional settings, where the tensor product of one-dimensional Gaussian Sobolev spaces or the dominating-mixed smoothness is considered, it is possible to extend our algorithm in such a way that the construction cost is still $\calO(n \log_2 n)$. One possibility is using rank-$1$ lattices as interpolation nodes, e.g. \cite{LH2003,MS2012,SSN2019,SN2020}. However, the convergence rate of approximation algorithms based on rank-$1$ lattice points deteriorates at least to $n^{-\alpha/2}$ for periodic functions \cite{BKUV2017}, and thus unfortunately it is not possible to obtain the same rate $n^{-\alpha}$. Another possibility is the fast Smolyak construction \cite{GH2014} combined with our algorithm $\Aa_n$, which may attain the optimal rate up to a logarithmic factor.
\end{remark}

\section{Spline algorithms} \label{sec:optimal}

In this section we consider a more refined decomposition of the real line.
By partitioning the line into intervals of unit length and having the number of points placed within each interval decrease exponentially fast as one moves away from the origin, we construct an algorithm that attains the optimal rate $n^{-\alpha}$ in $W_\rho^{\alpha,q}$ when $q = 2$. 
Naive version of this construction based on solving a linear system [see~\eqref{eq:spline-smoother}] requires $\calO(n^3)$ computational cost, and is similar to (and inspired by) those used in~\cite{DK2023} and~\cite[Section~3]{Karvonen2021}. 
Faster computation is possible in some special cases that we do not discuss here.

\subsection{Generic results}

Let $W^{\alpha, q}(\I)$ stand for the standard unweighted Sobolev space on an interval $\I$.
That is, each element of this space is $\alpha$ times weakly differentiable on $\I$ and the weak derivatives are in $L^q(\I)$.
For each $\nu \in \mathbb{N}$, let $A_\nu^{\I,1} \colon W^{\alpha, q}(\I_1) \to L^p(\I_1)$ be an approximation to $f \in W^{\alpha, q}(\I_1)$ based on $\nu$ function values on the interval $\I_1 = [0, 1]$.
We may interpret $A_\nu^{\I,1}(f)$ as an element of $L^p(\mathbb{R})$ that vanishes on $\mathbb{R} \setminus \I_1$.
We assume that there are positive constants $C_0$ and $\beta$ independent of $f$ or $\nu$ such that
\begin{equation} \label{eq:An0-generic}
        \| f - A_\nu^{\I,1}(f) \|_{L^p(\I_1)} \leq C_0 \nu^{-\beta}  \| f \|_{W^{\alpha, q}(\I_1)}
\end{equation}
for all $f \in W^{\alpha, q}(\I_1)$ and $\nu \in \mathbb{N}$. 
Let $\I_k = [k-1, k]$ for $k\in\Z$. 
By straightforward translation we obtain algorithms $A_\nu^{\I,k} \colon W^{\alpha, q}(\I_k) \to L^p(\I_k)$ such that
\begin{equation} \label{eq:Ank-generic}
        \| f - A_\nu^{\I,k}(f) \|_{L^p(\I_k)} \leq C_0 \nu^{-\beta}  \| f \|_{W^{\alpha, q}(\I_k)}
\end{equation}
for all $f \in W^{\alpha, q}(\I_k)$ and $\nu \in \mathbb{N}$.
For $m \in \N$ and $\nu_1, \ldots, \nu_m \in \N$ we set $n = 2(\nu_1 + \cdots + \nu_m)$.
Moreover, set $\nu_{-k} = \nu_k$.
Define the algorithm $\Ab_{n} \colon W^{\alpha, q}_\rho \to L^p_\rho$ that uses $n$ function values to approximate $f \in W^{\alpha, q}_\rho$ via
\begin{equation} \label{eq:An-optimal-def}
        \Ab_{n}(f) = \rho^{-1/p} \sum_{0 < \lvert k \lvert \leq m} A_{\nu_k}^{\I,k}( g_k ), \quad \text{ where } g_k = f|_{\I_k} \rho^{1/p}.
\end{equation}
Note that $\Ab_n(f)$ vanishes on $\mathbb{R} \setminus [-m, m]$.
We shall show that $\Ab_n$ inherits the convergence rate~\eqref{eq:An0-generic} of $A_\nu^{\I,1}$ if $m$ and $\nu_1, \ldots, \nu_m$ are selected properly.

\begin{proposition}[Generic error bound] \label{prop:optimal-generic}
Let $1\le p <q < \infty$, $\alpha\in\N$, and $f \in W^{\alpha, q}_\rho$. Then there are positive constants $C$ and $\delta$, which do not depend on $m$ or $\nu_1, \ldots, \nu_m$, such that
\begin{equation*}
        \| f - \Ab_{n}(f) \|_{L^p_\rho} \leq C \Bigg( \sum_{k=1}^m \e^{-\delta k^2} \nu_k^{-\beta} + \e^{-\delta m^2} \Bigg) .
\end{equation*}
\end{proposition}
\begin{proof}
The proof is similar to those in Section~\ref{sec:trig-interpolation}.
The error decomposes as
\begin{equation*}
        \| f - \Ab_n(f) \|_{L^p_\rho} \leq \sum_{0 < \lvert k \lvert \leq m} \| g_k - A_{\nu_k}^{\I,k}(g_k) \|_{L^p(\I_k)} + \left( \int_{ \R \setminus [-m, m]} \lvert f(x) \rvert^p \rho(x) \rd x \right)^{1/p}.
\end{equation*}
Let $\varepsilon \in (0, \frac{q-p}{pq})$.
Then Lemma~\ref{lemma:tail-bound} bounds the tail term as
\begin{equation*}
        \int_{ \R \setminus [-m, m]} \lvert f(x) \rvert^p \rho(x) \rd x = \int_{ \R \setminus [-m, m]} \lvert g(x) \rvert^p \rd x \leq C_{p, q, \varepsilon} \lVert g \rVert_\textup{decay} \, \e^{- \frac{(q-p)/pq - \varepsilon}{2} m^2 },
\end{equation*}
where $g = f \rho^{1/p}$ and $\lVert g \rVert_\textup{decay}$ is finite by Lemma~\ref{lem:decay}.
The error bound~\eqref{eq:Ank-generic} and Lemma~\ref{lem:g-norm-bound} yield
\begin{equation*}
\begin{split}
        \| g_k - A_{\nu_k}^{\I,k}(g_k) \|_{L^p(\I_k)} &\leq C_0 \nu_k^{-\beta} \lVert g_k \rVert_{W^{\alpha,q}(\I_k)} \leq C_0 C_{p,q,\alpha,\varepsilon} \, \rho^{\varepsilon/p}(\lvert k \rvert -1) \nu_k^{-\beta} \|f\|_{W^{\alpha,q}_{\rho}}.
\end{split}
\end{equation*}
The above bound does not depend on the sign of $k$ because we require that $\nu_{-k} = \nu_k$.
Hence there are positive constants $C$ and $\delta$ such that
\begin{equation*}
        \| f - \Ab_{n}(f) \|_{L^p_\rho} \leq C \Bigg( \sum_{k=1}^m \e^{-\delta k^2} \nu_k^{-\beta} + \e^{-\delta m^2} \Bigg),
\end{equation*}
which completes the proof.
\qed \end{proof}

\begin{theorem}[Error bound for exponentially decaying points] \label{Thm:optimal-generic}
Let $1\le p <q < \infty$, $\alpha\in\N$, and $m \in \mathbb{N}$. Set $\nu_k = 2^{m-k}$ for $k=1,\ldots,m$.
Then the algorithm $\Ab_{n}$ defined in~\eqref{eq:An-optimal-def} uses $n = 2(2^m - 1)$ function values and satisfies
\begin{equation*}
        \| f - \Ab_{n}(f) \|_{L^p_\rho} \leq C n^{-\beta}
\end{equation*}
for any $f \in W_\rho^{\alpha, q}$, where the constant $C$ does not depend on $n$.
\end{theorem}
\begin{proof}
The algorithm uses $n = 2\sum_{k=1}^m \nu_k = 2\sum_{k=1}^m 2^{m-k} = 2(2^m - 1)$ function values.
Proposition~\ref{prop:optimal-generic} yields
\begin{equation*}
        \| f - \Ab_{n}(f) \|_{L^p_\rho} \leq C \Bigg( \sum_{k=1}^m \e^{-\delta k^2} \nu_k^{-\beta}  + \e^{-\delta m^2} \Bigg) = C \Bigg( 2^{-\beta m} \sum_{k=1}^m \e^{-\delta k^2} 2^{\beta k}  + \e^{-\delta m^2} \Bigg).
\end{equation*}
Since $\sum_{k=1}^\infty \e^{-\delta k^2} 2^{\beta k} < \infty$ and $\e^{-\delta m^2}$ decays faster than $2^{-\beta m}$ as $m \to \infty$, the claim follows from
\ifpreprint
\begin{equation*}
        2^{-\beta m} \leq (2^m - 1)^{-\beta} = (n/2)^{-\beta}.
\end{equation*}
\else
\begin{equation*}
        2^{-\beta m} \leq (2^m - 1)^{-\beta} = (n/2)^{-\beta}. 
\end{equation*}
\RaggedLeft \qed
\fi
\end{proof}

To attain the optimal rate $n^{-\alpha}$ in $W_\rho^{\alpha,q}$ we need to exhibit a linear algorithm $A_\nu^{\I,1}$ that satisfies~\eqref{eq:An0-generic} for $\beta = \alpha$.
It is well known that for any $p, q \in [1, \infty]$ and $\nu \in \mathbb{N}$ one may construct an algorithm $A_\nu^{\I,1}$ that uses $\nu$ function values and has the error bound
\begin{equation} \label{Eq:Sobolev-rate}
        \| f - A_\nu^{\I,1}(f) \|_{L^p(\I_1)} \leq C_0 \nu^{-\alpha + (1/q - 1/p)_+}  \| f \|_{W^{\alpha, q}(\I_1)}
\end{equation}
for all $f \in W^{\alpha, q}(\I_1)$, where $(x)_+ = \max\{0, x\}$. 
See, for example, Section~1.3.11 in~\cite{Novak1988} or~\cite{KriegSonnleitner2023,NovakTriebel2006}.
Because we assume that $1 \leq p < q < \infty$, Theorem~\ref{Thm:optimal-generic} and~\eqref{Eq:Sobolev-rate} yield the following corollary, which gives the known upper bound on approximation error in $W_\rho^{\alpha,q}$~\cite[Theorem~3.3]{DK2023}.

\begin{corollary}[Convergence rate for sampling recovery]
Let $1\le p <q < \infty$ and $\alpha\in\N$. For each $n \in \mathbb{N}$ there is a linear algorithm of the form~\eqref{eq:An-optimal-def} that uses $n$ function values and satisfies
\begin{equation*}
        \| f - \Ab_{n}(f) \|_{L^p_\rho} \leq C n^{-\alpha},
\end{equation*}
for any $f \in W_\rho^{\alpha, q}$, where the constant $C$ does not depend on $n$.
\end{corollary}

\subsection{Spline smoothing}

The construction of an algorithm $A_\nu^{\I,1}$ satisfying~\eqref{Eq:Sobolev-rate} is particularly simple when $q = 2$, in which case $W^{\alpha,q}(\I_1)$ is a Hilbert space.

\begin{definition}[Spline smoother]
Let $W(\I_1)$ be any Hilbert space that is norm-equivalent to $W^{\alpha,2}(\I_1)$ and let $0 \leq x_1 < \cdots < x_\nu \leq 1$ be points on $\I_1$.
For any $\lambda > 0$, the minimiser $s_{f, \nu, \lambda}$ of
\begin{equation} \label{eq:minimisation-problem}
        \sum_{i=1}^\nu [ f(x_i) - s(x_i) ]^2 + \lambda \lVert s \rVert_{W(\I_1)}^2
\end{equation}
among $s \in W^{\alpha, 2}(\I_1)$ is unique.
This minimiser is called the \emph{spline smoother} to $f$ at $x_1, \ldots, x_\nu$.
\end{definition}

Because $\alpha \geq 1$, the Sobolev embedding theorem ensures that $W^{\alpha,2}(\I_1)$, and thus also $W(\I_1)$, is continuously embedded in the space of continuous functions on $\I_1$.
It follows that for each $x \in \I_1$ the point evaluation functional $f \mapsto f(x)$ is continuous on $W(\I_1)$.
By the Riesz representation theorem, for each $x \in \I_1$ there is a \emph{representer} $\ell_x \in W(\I_1)$ such that $f(x) = \langle f, \ell_x \rangle_{W(\I_1)}$ for every $f \in W(\I_1)$.
Using these representers we may define the \emph{reproducing kernel} of $W(\I_1)$ as
\begin{equation*}
        K(x, y) = \langle \ell_x, \ell_y \rangle_{W(\I_1)}.
\end{equation*}
This kernel is positive-semidefinite and satisfies $f(x) = \langle f, K(\cdot , x) \rangle_{W(\I_1)}$.

Some reproducing kernels are available in closed form.
The full scale of Sobolev spaces (also those of fractional order) is reproduced by the class of \emph{Matérn kernels} popular in machine learning and kriging~\cite{RasmussenWilliams2006,Stein1999}.
Let $l$ be a positive scaling parameter and $\mathcal{K}_\gamma$ the modified Bessel function of the second kind of order $\gamma$.
The Matérn kernel $K_\gamma$ of order $\gamma > 0$ given by
\begin{equation} \label{eq:matern-kernel}
        K_\gamma(x, y) = \frac{2^{1-\gamma}}{\Gamma(\gamma)} \bigg( \frac{\sqrt{2\gamma} \, \lvert x - y \rvert }{l} \bigg)^\gamma \mathcal{K}_\gamma \bigg( \frac{\sqrt{2\gamma} \, \lvert x - y \rvert }{l} \bigg) \eqqcolon \Phi_\gamma( \lvert x - y \rvert ),
\end{equation}
is a reproducing kernel for a space $W_\gamma(\I_1)$ that is norm-equivalent to $W^{\alpha,2}(\I_1)$ for $\alpha = \gamma +1/2$.
The curious coefficients $\sqrt{2\gamma}$ and $2^{1-\gamma} / \Gamma(\gamma)$ ensure that $K_\gamma$ tends pointwise to the \emph{Gaussian kernel} $\exp(-(x-y)^2/(2l^2))$ as $\gamma \to \infty$.
The norm-equivalence of $W_\gamma(\I_1)$ and $W^{\alpha,2}(\I_1)$ can be verified as follows.
It goes back at least to the work of Kimeldorf and Wahba~\cite{KimeldorfWahba1970} that a positive-semidefinite kernel of the form $K(x, y) = \Phi(\lvert x - y \rvert)$ for $\Phi \colon \mathbb{R} \to \mathbb{R}$ that is continuous and integrable is a reproducing kernel of a Hilbert space whose squared norm is proportional to
\begin{equation*}
        \int_{\mathbb{R}} \lvert \widehat{f}(\omega) \rvert^2 \, \widehat{\Phi}(\omega)^{-1} \rd \omega ,
\end{equation*}
where $\widehat{f}$ and $\widehat{\Phi}$ are the Fourier transforms of $f$ and $\Phi$ (see also Theorem~10.12 in~\cite{Wendland2005}).
Suppose that $\alpha = \gamma + 1/2$ is an integer.
Because the Fourier transform of the function $\Phi_\gamma$ in~\eqref{eq:matern-kernel} is proportional to $(c_1 + c_2 \omega^2)^{-(\gamma + 1/2)}$ for positive $c_1$ and $c_2$~\cite[Theorem~6.13]{Wendland2005}, the space $W_\gamma(\mathbb{R})$ in which $K_\gamma$ is reproducing on $\mathbb{R}$ has the squared norm
\begin{align*}
        \lVert f \rVert_{W_\gamma(\mathbb{R})}^2 \propto \int_{\mathbb{R}} \lvert \widehat{f}(\omega) \rvert^2 (c_1 + c_2 \omega^2)^{\gamma + 1/2} \rd \omega &= \sum_{\tau=0}^{\alpha} \binom{\alpha}{\tau} c_1^{\alpha-\tau} c_2^\tau \int_{\mathbb{R}} \lvert \widehat{f}(\omega) \rvert^2 \omega^{2\tau} \rd \omega \\
        &= \sum_{\tau=0}^{\alpha} \binom{\alpha}{\tau} c_1^{\alpha-\tau} c_2^\tau \int_{\mathbb{R}} |f^{(\tau)}(x)|^2 \rd x,
\end{align*}
where $f^{(\tau)}$ denotes the $\tau$-th order weak derivative and we have used the binomial theorem and Parseval's identity.
This norm is equivalent to the norm of the standard Sobolev space $W^{\alpha,2}(\mathbb{R})$ and by taking a restriction of $W_\gamma(\mathbb{R})$ on $\I_1$ we obtain a space $W_\gamma(\I_1)$ that is norm-equivalent to $W^{\alpha,2}(\I_1)$.

When one has access to the reproducing kernel $K$ of $W(\I_n)$, constructing the spline smoother is straightforward, though not necessarily computationally convenient.
It is a standard result~\cite[Section~1.3]{Wahba1990} that the minimiser of~\eqref{eq:minimisation-problem} takes the form
\begin{equation} \label{eq:spline-smoother}
        s_{f, \nu, \lambda}(x) = \sum_{i=1}^\nu a_i K(x, x_i).
\end{equation}
The coefficients $\boldsymbol{a} = (a_1, \ldots, a_\nu)$ are the solution to the linear system $(\boldsymbol{K}_\nu + \lambda \boldsymbol{I}_\nu) \boldsymbol{a} = \boldsymbol{f}_\nu$ of $\nu$ equations, where $\boldsymbol{K}_\nu = (K(x_i, x_j))_{i,j=1}^\nu$, $\boldsymbol{I}_\nu$ is the $\nu \times \nu$ identity matrix and $\boldsymbol{f}_\nu = (f(x_1), \ldots, f(x_\nu))$.
In learning theory this result is known as the \emph{representer theorem}~\cite{Scholkopf2001}.
Note that this linear system has a unique solution when $\lambda > 0$ because the matrix $\boldsymbol{K}_\nu$ is positive-semidefinite. 
Let $h_\nu = \max_{i=0,\ldots,\nu} \lvert x_{i+1} - x_i \rvert$, where we use the convention $x_0 = 0$ and $x_{\nu+1} = 1$.
By Proposition~3.6 in~\cite{WendlandRieger2005} (see~\cite{Arcangeli2007} for additional results), there is a positive constant $C_0$ independent of $f \in W^{\alpha,2}(\I_1)$ and the points $x_1, \ldots, x_\nu$ such that
\begin{equation} \label{eq:Wendland-Rieger}
        \lVert f - s_{f, \nu, \lambda} \rVert_{L^p(\I_1)} \leq C_0 h_\nu^\alpha \lVert f \rVert_{W^{\alpha,2}(\I_1)} 
\end{equation}
if $1 \leq p < 2$ and $\lambda \leq \nu^{-2\alpha}$.
This estimate provides an optimal rate of convergence for the following spline smoother based algorithm.
Namely, suppose that $0 = x_{\nu,0} < x_{\nu,1} < \cdots < x_{\nu,\nu} \leq x_{\nu,\nu+1} = 1$ for each $\nu \in \mathbb{N}$ and
\begin{equation} \label{eq:quasi-uniformity}
        \sup_{\nu \in \mathbb{N}} \, \max_{i=0,\ldots,\nu} \, \nu \, \lvert x_{\nu,i+1} - x_{\nu,i} \rvert < \infty.
\end{equation}
For example, the equispaced points
\begin{equation*}
    x_{\nu,i} = \frac{i}{\nu+1}
\end{equation*}
satisfy~\eqref{eq:quasi-uniformity} since $\nu \, \lvert x_{\nu,i+1} - x_{\nu,i} \rvert = \nu/(\nu+1) < 1$.
Set $\lambda_\nu = \nu^{-2\alpha}$ and let $s_{f,\nu,\lambda_\nu}$ be the spline smoother to $f$ at the points $x_{\nu,1}, \ldots, x_{\nu,\nu}$.
When
\begin{equation} \label{eq:A1n-spline}
        A_\nu^{\I,1}(f) = s_{f,\nu,\lambda_\nu},
\end{equation}
the algorithm $\Ab_{n}$ in~\eqref{eq:An-optimal-def} is a sum of $2m$ spline smoothers and uses $n$ function values.

\begin{corollary}[Convergence of a spline smoother algorithm] \label{cor:spline-convergence}
Let $1 \leq p < 2 \leq q$ and $\alpha\in\N$.
Let $m$ and $\nu_1, \ldots, \nu_m$ be as in Theorem~\ref{Thm:optimal-generic} and $A_n^{\I,1}$ as in~\eqref{eq:A1n-spline}.
Then the algorithm $\Ab_n$ defined in~\eqref{eq:An-optimal-def} uses $n = 2(2^m - 1)$ function values and satisfies
\begin{equation*}
        \| f - \Ab_n(f) \|_{L^p_\rho} \leq C n^{-\alpha}
\end{equation*}
for any $f \in W_\rho^{\alpha, q}$, where the constant $C$ does not depend on $n$.
\end{corollary}
\begin{proof}
Because the Gaussian density function $\rho$ defines a finite measure on $\R$, we have $L^q_\rho \subseteq L^2_\rho$ and consequently $W_\rho^{\alpha,q} \subseteq W_\rho^{\alpha,2}$.
The claim then follows from~\eqref{eq:Wendland-Rieger} and the fact that for points that satisfy~\eqref{eq:quasi-uniformity} we have $h_\nu \leq C_h n^{-1}$ for a positive constant $C_h$ independent of $\nu$.
\qed \end{proof}

In fact, the smoothness of the Sobolev space that is used to construct the spline smoothers $s_{f,\nu,\lambda_\nu}$ does \emph{not} have to coincide with the smoothness of $f$.
It is a consequence of the escape theorem of Narcowich, Ward and Wendland~\cite[Theorem~4.2]{Narcowich2006} that the estimate~\eqref{eq:Wendland-Rieger} remains valid even when the Hilbert space $W(\I_1)$ is norm-equivalent to $W^{\tilde{\alpha}, 2}(\I_1)$ for $\tilde{\alpha} \geq \alpha$.
From this we obtain the following corollary similar in spirit to Corollary~\ref{cor:alpha-free}.

\begin{corollary}[Misspecified smoothness] \label{cor:smoothness-misspecification}
Let $\tilde{\alpha} \geq \alpha$ and consider the setting of Corollary~\ref{cor:spline-convergence} but suppose that the spline smoothers $s_{f,\nu,\lambda_\nu}$ are constructed using a Hilbert space $W(\I_1)$ norm-equivalent to $W^{\tilde{\alpha}, 2}(\I_n)$.
Then
\begin{equation*}
        \| f - \Ab_n(f) \|_{L^p_\rho} \leq C n^{-\alpha}
\end{equation*}
for any $f \in W_\rho^{\alpha, q}$, where the constant $C$ does not depend on $n$.
\end{corollary}

\begin{remark}[Fractional smoothness]
Because the results in~\cite{Arcangeli2007,Narcowich2006,WendlandRieger2005} that we have used apply also to fractional Sobolev spaces, Corollaries~\ref{cor:spline-convergence} and~\ref{cor:smoothness-misspecification} generalise to the fractional setting if the weighted Sobolev space $W^{\alpha,q}_\rho$ is defined appropriately for $\alpha \notin \mathbb{N}$.
\end{remark}

\begin{remark}[Higher dimensions]
    By replacing the unit intervals $\I_k$ with suitably selected unit cubes as in~\cite{DK2023}, Corollaries~\ref{cor:spline-convergence} and~\ref{cor:smoothness-misspecification} could be generalised also to the $d$-dimensional weighted isotropic Sobolev space $W_\rho^{\alpha, q}(\mathbb{R}^d)$ consisting of all functions whose Gaussian-weighted mixed weak derivatives of total order at most $\alpha$ are $q$-integrable.
    The estimate~\eqref{eq:Wendland-Rieger} remains valid when $\I_1$ is replaced with a $d$-dimensional unit cube $\I_1^d$ and $h_\nu$ with the $d$-dimensional fill-distance $\max_{\boldsymbol{x} \in \I_1^d} \min_{i=1,\ldots,\nu} \lVert \boldsymbol{x} - \boldsymbol{x}_i \rVert$, where $\lVert \cdot \rVert$ is the Euclidean norm.
\end{remark}

\section{Optimality}\label{sec:opt}
We conclude this paper by discussing optimality, i.e., general lower bounds for algorithms which are based on function values. For $p<q$, the optimal rate is derived by Kuo, Plaskota, and Wasilkowski \cite[Theorem~3]{KPW2016}. They consider a more general class of weights, which is applicable to ours. 
Their result states that the error of any algorithm $A_n$ that only uses $n$ function evaluation has the lower bound
\[
\sup_{\substack{f\in W^{\alpha,q}_\rho \\
\|f\|_{W^{\alpha,q}_\rho}\le1}} \|f-A_n(f)\|_{L_\rho^p} \ge C_{p,q,\alpha}n^{-\alpha},
\]
for $1\le p<q$. In this sense, our methods achieving error convergence $n^{-\alpha}$ are optimal.
In the same paper, they propose a piecewise Lagrange approximation that achieves the optimal rate \cite[Theorem~2]{KPW2016}.

Even though our algorithms do not cover the extreme case $p=q$, this has been recently studied by D\~{u}ng \cite[Theorem~3.4]{D2025}.
In this case the optimal rate degenerates to $n^{-\alpha/2}$. D\~{u}ng \cite{D2025} also covers extreme cases, such as $q=1$ or $q=\infty$, deriving the convergence rate $n^{-\alpha/2+(1-1/p)/2}$ for $q=1$ and $n^{-\alpha/2+1/(2p)}$ for $q=\infty$. These rates can be achieved by his B-spline algorithms \cite[Theorem~2.2 and 2.4]{D2025}.

In this paper, we focused on implementable and efficient algorithms achieving the optimal approximation rate. In practice, we often do not know the smoothness $\alpha$ of the target function and we have access to only function values. We explored the case when $\alpha$ is not known exactly in Corollaries~\ref{cor:alpha-free} and \ref{cor:smoothness-misspecification}. It is an interesting question of whether optimal algorithms can be constructed for the extreme cases of $(p,q)$ without knowledge of $\alpha$. We leave this to future study.

\section*{Acknowledgement}
The authors were supported by the Research Council of Finland (decisions 338567, 348503, 359181, 359183, and 368086).
TK acknowledges the research environment provided by ELLIS Institute Finland. We thank the anonymous reviewers for valuable comments.

\bibliographystyle{plain}
\bibliography{reference.bib}

\end{document}